\documentclass{patmorin}
\usepackage{pat}
\usepackage{paralist}
\usepackage{dsfont}  % for \mathds{A}
\usepackage[utf8x]{inputenc}
\usepackage{skull}
\usepackage{paralist}
\usepackage{graphicx}
\usepackage[noend]{algorithmic}

\usepackage[normalem]{ulem}
\usepackage{cancel}
\usepackage{enumitem}

\usepackage{todonotes}

\usepackage[longnamesfirst,numbers,sort&compress]{natbib}

\usepackage[mathlines]{lineno}
\newcommand*\patchAmsMathEnvironmentForLineno[1]{%
 \expandafter\let\csname old#1\expandafter\endcsname\csname #1\endcsname
 \expandafter\let\csname oldend#1\expandafter\endcsname\csname end#1\endcsname
 \renewenvironment{#1}%
    {\linenomath\csname old#1\endcsname}%
    {\csname oldend#1\endcsname\endlinenomath}}%
\newcommand*\patchBothAmsMathEnvironmentsForLineno[1]{%
 \patchAmsMathEnvironmentForLineno{#1}%
 \patchAmsMathEnvironmentForLineno{#1*}}%
\AtBeginDocument{%
\patchBothAmsMathEnvironmentsForLineno{equation}%
\patchBothAmsMathEnvironmentsForLineno{align}%
\patchBothAmsMathEnvironmentsForLineno{flalign}%
\patchBothAmsMathEnvironmentsForLineno{alignat}%
\patchBothAmsMathEnvironmentsForLineno{gather}%
\patchBothAmsMathEnvironmentsForLineno{multline}%
}

\DeclareMathOperator{\afcn}{\dot{\chi}_\pi}

\title{\MakeUppercase{$2\times n$ Grids have Unbounded Anagram-Free Chromatic Number}\thanks{This research was partly funded by NSERC.}}
\author{Saman Bazarghani%
    \thanks{Department of Computer Science and Electrical Engineering, University of Ottawa}\qquad
    Paz Carmi%
    \thanks{Ben-Gurion University of the Negev}\qquad
    Vida Dujmović\footnotemark[2]\qquad
    Pat Morin%
    \thanks{School of Computer Science, Carleton University}}

\date{}

\DeclareMathOperator{\hist}{h}

\usepackage{tabularx}

\begin{document}

% \begin{titlepage}
\maketitle

\begin{abstract}
    We show that anagram-free vertex colouring a $2\times n$ square grid requires a number of colours that increases with $n$.  This answers an open question in Wilson's thesis and shows that even graphs of pathwidth 2 do not have anagram-free colourings with a bounded number of colours.
\end{abstract}
% \end{titlepage}

% \pagenumbering{roman}
% \tableofcontents
%
% \newpage
% \pagenumbering{arabic}

\section{Introduction}

Two strings $s$ and $t$ are said to be \emph{anagrams} of each other if $s$ is a permutation of $t$.  A single string $s:=s_1,\ldots,s_{2r}$ is \emph{anagramish} if its first half $s_1,\ldots,s_r$ and its second half $s_{r+1},\ldots,s_{2r}$ are anagrams of each other.  A string is \emph{anagram-free} if it does not contain an anagramish substring.

In 1961, \citet{erdos:some} asked if there exists arbitrarily long strings over an alphabet of size $4$ that are anagram-free.\footnote{This was an incredibly prescient question since it is not at all obvious that there exist arbitrarily long anagram-free strings over any finite alphabet. The only justification for choosing the constant $4$ is that a short case analysis rules out the possibility of length-$8$ anagram-free strings over an alphabet of size $3$.}  In 1968 \citet{evdokimov:strongly,evdokimov:strongly2} showed the existence of arbitrarily long anagram-free strings over an alphabet of size $25$ and in 1971 \citet{pleasants:non-repetitive} showed an alphabet of size $5$ is sufficient.  Erd\H{o}s's question was not fully resolved until 1992, when \citet{keranen:abelian} answered it in the affirmative.

A path $v_1,\ldots,v_{2r}$ in a graph $G$ is \emph{anagramish} under a vertex $c$-colouring $\phi:V(G)\to\{1,\ldots,c\}$ if $\phi(v_1),\ldots,\phi(v_{2r})$ is an anagramish string.  The colouring $\phi$ is an \emph{anagram-free colouring} of $G$ if no path in $G$ is anagramish under $\phi$.  The minimum integer $c$ for which $G$ has an anagram-free vertex $c$-colouring is called the \emph{anagram-free chromatic number} of $G$ and is denoted by $\afcn(G)$.

For a graph family $\mathcal{G}$, $\afcn(\mathcal{G}):=\max\{\afcn(G):G\in\mathcal{G}\}$ or $\afcn(\mathcal{G}):=\infty$ if the maximum is undefined. The results on anagram-free strings discussed in the preceding paragraph can be interpreted in terms of $\afcn(\mathcal{P})$ where $\mathcal{P}$ is the family of all paths.  Slightly more complicated than paths are trees.  \citet{wilson.wood:anagram-free} showed that $\afcn(\mathcal{T})=\infty$ for the family $\mathcal{T}$ of trees and \citet{kamcev.luczak.ea:anagram-free} showed that $\afcn(\mathcal{T_2})=\infty$ even for the family $\mathcal{T}_2$ of binary trees.

One positive result in this context is that of
\citet{wilson.wood:anagram-free}, who showed that every tree $T$ of pathwidth $p$ has $\afcn(T)\le 4p+1$. Since trees are graphs of treewidth $1$ it is natural to ask if this result can be extended to show that every graph $G$ of treewidth $t$ and pathwidth $p$ has $\afcn(G)\le f(t,p)$ for some $f:\N^2\to\N$.  \citet{carmi.dujmovic.ea:anagram-free} showed that such a generalization is not possible for any $t\ge 3$ by giving examples of $n$-vertex graph of pathwidth $3$ (and treewidth $3$) with $\afcn(G)\in\Omega(\log n)$.  The obvious remaining gap left by these two works is graphs of treewidth $2$. Our main result is to show that $G_n$, the $2\times n$ square grid has $\afcn(G_n)\in\omega_n(1)$. Since $G_n$ has pathwidth 2, we have:

\begin{thm}\label{main_vertex}
    For every $c\in\N$, there exists a graph of pathwidth $2$ that has no anagram-free vertex $c$-colouring.
\end{thm}

\citet[Section 7.1]{wilson:anagram-free} conjectured that $\afcn(G_n)\in\omega_n(1)$, so this work confirms this conjecture.  Prior to the current work, it was not even known if the family of $n\times n$ square grids had anagram-free colourings using a constant number of colours.

In a larger context, this lower bound gives more evidence that, except for a few special cases (paths \cite{evdokimov:strongly,pleasants:non-repetitive,keranen:abelian}, trees of bounded pathwidth \cite{wilson.wood:anagram-free}, and highly subdivided graphs \cite{wilson.wood:anagram-free2}), the qualitative behaviour of anagram-free chromatic number is not much different than that of treedepth/centered colouring \cite{nesetril.ossona:tree-depth}.  Very roughly: For most graph classes, every graph in the class has an anagram-free colouring using a bounded number of colors precisely when every graph in the class has a colouring using a bounded number of colours in which every path contains a colour that appears only once in the path.

% \todo[inline]{Statement of vertex colouring result (hopefully)}
%
% An \emph{edge $c$-colouring} $\varphi:E(G)\to\{1,\ldots,c\}$ is \emph{anagram-free} if $G$ has no path $v_0,\ldots,v_{2r}$ such that $\varphi(v_0v_1),\ldots,\varphi(v_{2r-1}v_{2r})$ is anagramish.  We also give an edge colouring version of \cref{main_vertex}.
%
% \begin{thm}\label{main_edge}
%     For every $c\in\N$, there exists a graph of $G$ of pathwidth $2$ that has no anagram-free edge $c$-colouring.
% \end{thm}

% \todo[inline]{Relation to Existing Work}

The remainder of this paper is organized as follows: \Cref{near_anagram_statement} gives some definitions and states a key lemma that shows that, under a certain periodicity condition, every sufficiently long string contains a substring that is $\epsilon$-close to being anagramish.
% In \cref{edge_colourings} we prove \cref{main_edge}.
In \cref{vertex_colourings} we prove \cref{main_vertex}.  In \cref{near_anagram_proof} we prove the key lemma. \cref{reflections} concludes with some final remarks about the (non-)constructiveness of our proof technique.

\section{Periodicity in Strings}
\label{near_anagram_statement}

An \emph{alphabet} $\Sigma$ is a finite non-empty set.  A \emph{string over $\Sigma$} is a (possibly empty) sequence $s:=s_1,\ldots,s_n$ with $s_i\in\Sigma$ for each $i\in\{1,\ldots,n\}$ The \emph{length} of $s$ is the length, $n$, of the sequence. For an integer $k$, $\Sigma^k$ (the $k$-fold cartesian product of $\Sigma$ with itself) is the set of all length-$k$ strings over $\Sigma$.  The \emph{Kleene closure} $\Sigma^*:=\bigcup_{k=0}^\infty \Sigma^k$ is the set of all strings over $\Sigma$.  For each $1 \le i \le j\le n+1$, $s[i\mathbin{:}j]:=s_i,s_{i+1},\ldots,s_{j-1}$ is called a \emph{substring} of $s$. (Note the convention that $s[i\mathbin{:}j]$ does not include $s_j$, so $s[i\mathbin{:}j]$ has length $j-i$.)

Let $s:=s_1,\ldots,s_n$ be a string over an alphabet $\Sigma:=\{s_i:i\in\{1,\ldots,n\}\}$ and, for each $a\in \Sigma$, define $\hist_a(s):=|\{i\in\{1,\ldots,n\}:s_i=a\}|$.  The \emph{histogram} of $s$ is the integer-valued $|\Sigma|$-vector $\hist(s):=(\hist_a(s):a\in\Sigma)$ indexed by elements of $\Sigma$.  Observe that a string $s_1,\ldots,s_{2r}$ is anagramish if and only if $\hist(s_1,\ldots,s_r)=\hist(s_{r+1},\ldots,s_{2r})$ or, equivalently, $\hist(s_1,\ldots,s_r)-\hist(s_{r+1},\ldots,s_{2r})=\boldsymbol{0}$.
% Let $\|\cdot\|_1$ denote the L1 norm of a vector.\footnote{The L1 norm of a real-valued $d$-vector $\boldsymbol{v}:=(v_1,\ldots,v_d)$, $\|\boldsymbol{v}\|_1=\sum_{i=1}^d|v_i|$.}
For each $a\in\Sigma$, let $\tau_a:=|\hist_a(s_1,\ldots,s_r)-\hist_a(s_{r+1},\ldots,s_{2r})|$.
Then $\tau(s):=\sum_{a\in\Sigma}\tau_a(s)$ is a useful measure of how far a string is from being anagramish and $\tau(s)=0$ if and only if $s$ is anagramish.

A string $s:=s_1,\ldots,s_n$ is $\ell$-periodic if each length-$\ell$ substring of $s$ contains every character in $\Sigma:=\bigcup_{i=1}^n s_i$.  We make use of the following lemma, which states that every sufficiently long $\ell$-periodic string contains a long substring that is $\epsilon$-close to being anagramish.

\begin{lem}\label{near_anagram_fourier}
    For each $r_0,\ell\in\N$ and each $\epsilon>0$, there exists a positive integer $n$ such that each $\ell$-periodic string $s_1,\ldots,s_n$ contains a substring $s:=s_{i+1},\ldots,s_{i+2r}$ of length $2r \ge 2r_0$ such that $\tau(s)\le \epsilon r$.
\end{lem}

The proof of \cref{near_anagram_fourier} is deferred to \cref{near_anagram_proof}.  We now give some intuition as to how it is used.  The process of checking if a string is anagramish is often viewed as finding common terms in the first and second halves and crossing them both out.  If this results in a complete cancellation of all terms, then the string is an anagram.  \cref{near_anagram_fourier} tells us that we can always find a long substring $s$ where, after exhaustive cancellation, only an $\epsilon$-fraction of the original terms remain.  Informally, the substring $s$ is $\epsilon$-close to being anagramish.

\cref{near_anagram_fourier} says that, if up to $\epsilon r$ terms in each half of $s$ were each allowed to cancel two terms each in the other half of $s$, then it would be possible to complete the cancellation process.  To achieve this type of one-versus-two cancellation in our setting, we decompose our coloured pathwidth-$2$ graph into pieces of constant size.  The vertices in each piece can be covered with one path or partitioned into two paths.  In this way an occurrence $H_z$ of a particular coloured piece in one half can be matched with two like-coloured pieces $H_x$ and $H_y$ in the other half. We construct a single path $P$ that contains all vertices in $H_z$ and only half the vertices in each of $H_x$ and $H_y$.  In this way, the colours of vertices $P\cap H_z$ can cancel the colours of the vertices in $P\cap(H_x\cup H_y)$.

Since \cref{near_anagram_fourier} requires that the string $s$ be $\ell$-periodic, the following lemma will be helpful in obtaining strings that can be used with \cref{near_anagram_fourier}.

\begin{lem}\label{periodicity}
    Let $\Sigma$ be an alphabet and let $P:\Sigma^*\to\{0,1\}$ be a function such that,
    \begin{compactenum}[({A}1)]
        \item if $P(s')=0$ for some substring $s'$ of $s$ then $P(s)=0$; and
        \item for each $n\in\N$, there exists at least one $s\in \Sigma^n$ such that $P(s)=1$.
    \end{compactenum}
    Then there exists $\ell\in\N$ and $\Xi\subseteq\Sigma$ such that, for each $n\in\N$,
    \begin{compactenum}[(C1)]
        \item there exists $s\in\Xi^n$ such that $P(s)=1$; and
        \item every string in $\{s\in\Xi^n:P(s)=1\}$ is $\ell$-periodic.
    \end{compactenum}
\end{lem}

\begin{proof}
    Take $\Xi$ to be a minimal subset of $\Sigma$ such that there exists $s\in\Xi^n$ with $P(s)=1$, for each $n\in\N$.  Such a $\Xi$ exists by (A2) and the fact that $\Sigma$ is finite. By definition, $\Xi$ satisfies (C1) so we need only show that it also satisifies (C2).  If $|\Xi|=1$ then we are done since every string over a $1$-character alphabet is $1$-periodic.

    For $|\Xi|>1$, the minimality of $\Xi$ implies that, for any $a\in\Xi$, there exist $\ell_a\in\N$ such that $P(s)=0$ for each $s\in\Xi^{\ell_a}$.  Therefore (A1) implies that, for each $\ell\ge\ell_a$, $P(s)=0$ for each $s\in\Xi^{\ell}$. Set $\ell:=\max\{\ell_a:a\in \Xi\}$.  Now (A1) implies that, for every $s\in\Xi$, every length $\ell$-substring of $s$ contains every character in $\Xi$, so $s$ is $\ell$-periodic.
\end{proof}

\section{Proof of \cref{main_vertex}}
\label{vertex_colourings}

For each $n\in\N$, let $G_n$ be the $2\times n$ square grid with top row $a_0,\ldots,a_{n-1}$ and bottom row $b_0,\ldots,b_{n-1}$ (see \cref{g_n}). For convenience, we let $G:=G_\infty$.  For each $i,j\in\N$ with $i\le j$, define $G[i\mathbin{:}j]:=G[\bigcup_{k=i}^{j-1}\{a_k,b_k\}]$ and we call $G[i\mathbin{:}j]$ a \emph{$(j-i)$-block}. Note that each $t$-block $G[i\mathbin{:}i+t]$ is isomorphic to $G_t$ with the mapping $f:G_t\to V(G[i\mathbin{:}i+t])$ given by $f_i(a_{j}):=a_{i+j}$ and $f_i(b_{j}):=b_{i+j}$ for each $j\in\{0,\ldots,t-1\}$.

\begin{figure}
    \centering{
        \includegraphics{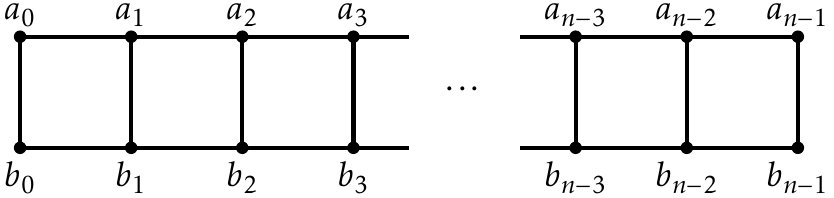}
    }
    \caption{The graph $G_n$}
    \label{g_n}
\end{figure}

For each $n\in\N$, let $\Phi_{c,n}$ be the set of all $c^{2n}$ functions $\phi:V(G_n)\to\{1,\ldots,c\}$, i.e., all vertex $c$-colourings of $G_n$.
Given some $\phi\in\Phi_{c,n}$, each $t$-block $G[i\mathbin{:}i+t]$ defines a vertex colouring $\phi[i\mathbin{:}i+t]\in \Phi_{c,t}$ of $G_t$ defined as $\phi[i\mathbin{:}i+t](v):=\phi(f_i(v))$ for each $v\in V(G[i\mathbin{:}i+t])$.

Our strategy will be to break $G_n$ up into small pieces using $4$-blocks that all have the same colouring. Observe that any string $s:=\phi_0,\ldots,\phi_{b-1}\in(\Phi_{c,4})^b$ defines a vertex $c$-colouring $\phi_s$ of the graph $G_{4b}$ where $\phi_s[4j\mathbin{:}4j+4]:=\phi_j$ for each $j\in\{0,\ldots,b-1\}$.  Indeed, this is a bijection between $c$-colourings of $G_{4b}$ and strings in $\Phi_{c,4}^b$.

\begin{lem}\label{breakers}
    If $\afcn(G_n)\le c$ for each $n\in\N$ then there exists $\Xi_{c,4}\subseteq\Phi_{c,4}$ such that, for each $b\in\N$ there exists $s\in(\Xi_{c,4})^b$ such that $\phi_s$ is an anagram-free vertex colouring of $G_{4b}$ and every such $s$ is $\ell$-periodic.
\end{lem}

\begin{proof}
    For any $s\in(\Phi_{c,4})^*$, let $P(s):=1$ if $\phi_s$ is an anagram-free colouring of $G_{4|s|}$ and let $P(s):=0$ otherwise.  Then $P$ has property (A1) of \cref{periodicity} since any substring of $s'$ of $s$ defines a colouring of $G_{4|s'|}$ that appears in the colouring of $G_{4|s|}$; if the colouring $\phi_{s'}$ of $G_{4|s'|}$ is not anagram-free then neither is the colouring $\phi_s$ of $G_{4|s|}$. By assumption, $\afcn(G)\le c$, so $G_{4b}$ has some anagram-free vertex $c$-colouring, so $P$ also satisfies property (A2) of \cref{periodicity}.  The result now follows from  \cref{periodicity}.
\end{proof}

Let $\Xi_{c,4}$ and $\ell$ be defined as in \cref{breakers}. Let $\phi^*$ be an arbitrary element of $\Xi_{c,4}$ and let $\Sigma:=\{(k,\phi):k\in\{1,\ldots,\ell\},\, \phi\in\Phi_{c,4k}\}$. For a string $s:=(k_1,\phi_1),\ldots,(k_r,\phi_r)\in\Sigma^*$, define $n_{s,i}:=4(i+\sum_{j=1}^i k_j)$, define $n_s:=n_{s,r}$.  Fix a vertex $c$-colouring $\phi^*\in \Phi_{c,4}$ of $G_4$ and define the colouring $\phi_s$ of $G_{n_s}$ as follows:
\begin{compactenum}
    \item for each $i\in\{0,\ldots,r\}$, $\phi[n_{s,i},n_{s,i+4}]:=\phi^*$; and
    \item for each $i\in\{1,\ldots,r\}$, $\phi[n_{s,i-1}+4,n_{s,i}]:=\phi_i$.
\end{compactenum}
See \cref{bigexample}.  In words, $G_{n_s}$ is decomposed into blocks each of whose length is a multiple of $4$.  There are \emph{colourful} blocks of lengths $4k_1,4k_2,\ldots,4k_{|s|}\le 4\ell$ and these are interleaved with \emph{boring} blocks, each of length $4$.  The colourful blocks have their vertex colours determined by $s$. The boring blocks are all coloured the same way, by $\phi^*$.

\begin{figure}
    \centering{\includegraphics{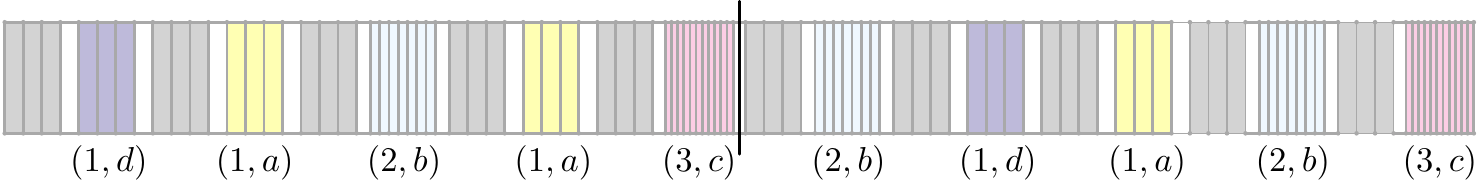}}
    \caption{A string $s\in\Sigma^{10}$ and its corresponding coloured graph $G_{n_s}$.}
    \label{bigexample}
\end{figure}

% If $\afcn(G_n)\le c$ for each $n\in\N$ then \cref{breakers} implies that, for each $n\in\N$ there exists a string $s\in\Sigma^n$ such that $\phi_s$ is an anagram-free colouring of $G_{n_s}$.

Define the function $P:\Sigma^*\to\{0,1\}$ so that $P(s):=1$ if $\phi_s$ is an anagram-free colouring of $G_{n_s}$ and $P(s):=0$ otherwise.  Observe that any substring $s'$ of $s$ defines a colouring $\phi_{s'}$ of $G_{n_{s'}}$ that appears in the colouring $\phi_s$ of $G_{n_s}$. Therefore $P$ satisfies property (A1) of \cref{periodicity}.  Furthermore, if $\afcn(G_n)\le c$ for each $n\in\N$, then \cref{breakers} implies that there exists a string $s\in\Sigma^{b}$ with $P(s)=1$ for each $b\in\N$.  Therefore $P$ satisfies property (A2) of \cref{periodicity}.  Therefore, \cref{periodicity} implies the following result:

\begin{lem}\label{block_colouring}
    If $\afcn(G_n)\le c$ for each $n\in\N$ then there exists $\ell\in\N$ such that, for each $k\in\N$ there exists an $\ell$-periodic string $s\in\Sigma^{k}$ such that $\phi_s$ is an anagram-free vertex colouring of $G_{n_s}$.
\end{lem}

\Cref{near_anagram_fourier,block_colouring} immediately imply:

\begin{lem}\label{near_anagram_graph}
    If $\afcn(G_n)\le c$ for each $n\in\N$ then, for each $r_0\in\N$ and $\epsilon>0$, there exists $r\ge r_0$ and a string $s:=s_1,\ldots,s_{2r}\in\Sigma^{2r}$ with $\tau(s)\le\epsilon r$ such that $\phi_s$ is an anagram-free vertex colouring of $G_{n_s}$.
\end{lem}

\begin{proof}
    By \cref{block_colouring}, for each $k\in\N$ there exists an $\ell$-periodic string $s'\in\Sigma^k$ such that $\phi_{s'}$ is an anagram-free vertex colouring of $G_{n_{s'}}$. Applying \cref{near_anagram_fourier} to $s'$ proves the existence of the desired string $s$.
\end{proof}

\cref{near_anagram_graph} shows the existence of colourings of $G_n$ for arbitrarily large values of $n$ that are defined by strings that are $\epsilon$-close to being anagramish.  The last step, done in the next lemma, is to show that there is enough flexibility when constructing a path in $G_n$ that we can find a path that has an anagramish colour sequence.

\begin{lem}\label{anagramish_path}
    For any $\ell\in\N$, there exists $\epsilon>0$ and $r_0\in\N$ such that, for any integer $r\ge r_0$ and any $\ell$-periodic $s:=s_1,\ldots,s_{2r}\in\Sigma^{2r}$ with $\tau(s)\le\epsilon r$, the graph $G_{n_s}$ contains a path $P=v_1,\ldots,v_{2m}$ such that $\phi_s(v_1),\ldots,\phi_s(v_{2m})$ is anagramish.
\end{lem}

\begin{proof}
    For each $a\in\Sigma$ define $\delta_a := \hist_a(s_1,\ldots,s_r)-\hist_a(s_{r+1},\ldots,s_{r2})$ and define sets $A_a\subseteq\{i\in \{1,\ldots,r-1\}: s_i=a\}$ and $B_a\subseteq\{i\in\{r+1,\ldots,2r-1\}:s_i=a\}$ as follows:
    \begin{compactenum}
        \item If $\delta_a=0$ then $A_a=B_a=\emptyset$.
        \item If $\delta_a>0$ then $|A_a|=2\delta_a=2|B_a|$.
        \item If $\delta_a<0$ then $|A_a|=\delta_a=\tfrac{1}{2}|B_a|$.
    \end{compactenum}
    Let $A:=\bigcup_{a\in\Sigma} A_a$ and $B:=\bigcup_{a\in\Sigma} B_a$.
    The sets $A_a$ and $B_a$ are chosen so that they satisfy the following \emph{global independence constraint}:  There is no pair $i,j\in A\cup B$ such that $i-j=1$.  To see that this is possible, first observe that, because $r$ is not in $A$ or $B$ we need only concern ourselves with pairs where both $i,j\in\{1,\ldots,r-1\}$ or pairs where both $i,j\in\{r+1,\ldots,2r-1\}$.  Thus, we can choose the elements of $A_a$, for each $a\in\Sigma$ and then independently choose the elements of $B_a$, for each $a\in\Sigma$.

    We show how to choose the elements of $A_a$ for each $a\in\Sigma$.  The same method works for choosing the elements in $B_a$. Observe that, because $s$ is $\ell$-periodic, $|\{i\in \{1,\ldots,r-1\}:s_i=a\}|\ge \lfloor(r-1)/\ell\rfloor$ for each $a\in\Sigma$.  This allows us to greedily choose the elements in $A_a$ for each $a\in\Sigma$. At each step we simply avoid choosing $i$ if $i-1$ or $i+1$ have already been chosen in some previous step.  At any step in the process, at most $\epsilon r$ elements have already been chosen in previous steps and each of these eliminates at most $2$ options.  Therefore, there will always be an element available to choose, provided that $2\epsilon r < \lfloor(r-1)/\ell\rfloor$.  In particular, for any $r\ge r_0\ge 2\ell$, $\epsilon < 1/(4\ell)$ works.

    We now construct the path $P$ in a piecewise fashion.  Refer to \cref{path_construction}.  For each $i\in\{1,\ldots,2r\}$, let $H_i:=G[n_{s,i}:n_{s,i+1}-4]$. The subgraph $H_1,\ldots,H_{2r}$ are what is referred to above as colourful blocks.  The colouring of $V(H_i)$ by $\phi_s$ is defined by $s_i$.
    \begin{figure}
        \centering{
            \includegraphics{figs/bigexample-1} \\
            \includegraphics{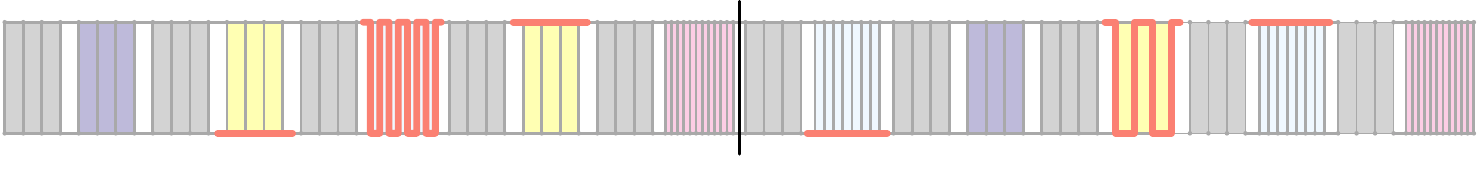} \\
            \includegraphics{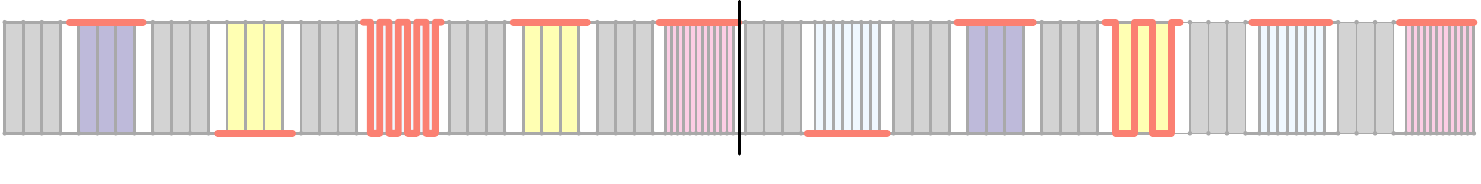} \\
            \includegraphics{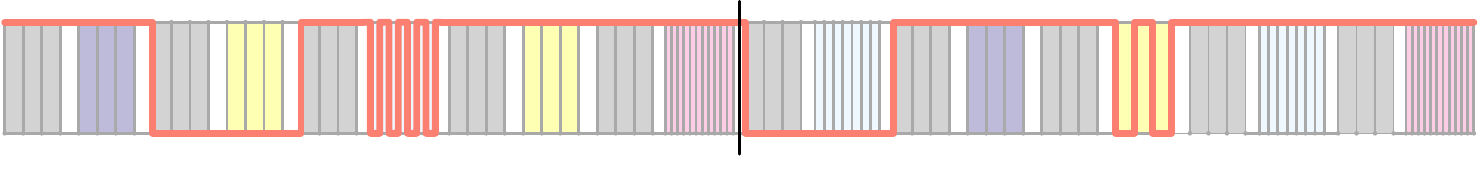}
        }
        \caption{Constructing the anagramish path $P$:
            (1) top and bottom paths are matched with zig-zag paths;
            (2) all remaining colourful blocks receive top paths;
            (3) all boring blocks receive top, updown, or downup paths.
        }
        \label{path_construction}
    \end{figure}
    \begin{compactenum}
        \item For each $a\in\Sigma$ such that $\delta_a>0$, group the elements of $A_a$ into pairs.  For each pair $(i,j)$, $P$ contains the path through the top row of $H_i$ and the path through the bottom row of $H_j$.  For each element $i\in B_a$, $P$ contains the zig-zag path with both endpoints in the top row of $H_i$ and that contains every vertex of $H_i$.  (Note that the zig-zag path begins at the top and ends at the bottom row because $H_i$ is a $t$-block for $t$ a multiple of $4$; in particular, $t$ is even.)

        \begin{figure}
            \centering{
                \begin{tabular}{ccc}
                    \includegraphics{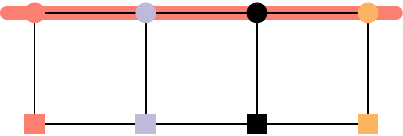} &
                    \includegraphics{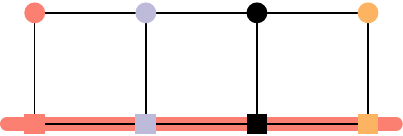} &
                    \includegraphics{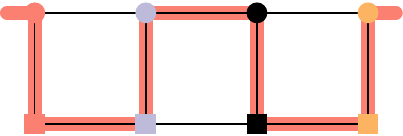}
                \end{tabular}
            }
            \caption{Subpaths of $P$ through colourful blocks: A top path and a bottom path contribute the same amount as a single zig-zag path.}
        \end{figure}

        \item For each $a\in\Sigma$ such that $\delta_a<0$ we proceed symmetrically to the previous case, but reversing the roles of $A_a$ and $B_a$.  Specifically, we group the elements of $B_a$ into pairs.  For each pair $(i,j)$, $P$ contains the path through the top row of $H_i$ and the path through the bottom row of $B_j$.  For each element $i\in A_a$, $P$ contains the zig-zag path with both endpoints in the top row of $H_i$ and that contains every vertex of $H_i$.

        \item For each $i\in\{1,\ldots,2r\}\setminus\bigcup_{a\in\Sigma}(I_a\cup B_a)$, $P$ contains the top row of $H_i$.
    \end{compactenum}
    The rules above define the intersection, $P_i$, of $P$ with each colourful block $H_i$ of $G_{n_s}$.  If $P_i$ is the path through the bottom (top) row of $H_i$ then we call $H_i$ a \emph{bottom (top) block}.  If $P_i$ is the zig-zag path that contains every vertex of $H_i$ then we call $H_i$ a \emph{zig-zag block}.  Note that $\sum_{a\in\Sigma} \delta_a = 0$ and this implies that the number of bottom blocks among $H_1,\ldots,H_{r-1}$ is the same as the number of bottom blocks among $H_{r+1},\ldots,H_{2r}$.  Indeed, this number is exactly $\beta:=\tfrac{1}{2}\sum_{a\in\Sigma} |\delta_a|=\tfrac{1}{2}\tau(s)$.

    We now define how $P$ behaves for the boring blocks, that we name $Q_0,\ldots,Q_{2r-1}$. The first boring block $Q_0$ comes immediately before $H_1$. Each boring block $Q_j$, for $j\in\{1,\ldots,2r-1\}$ comes immediately after $H_j$ and immediately before $H_{j+1}$.  In almost every case, $P$ uses the path through the top row of $Q_j$.  The only exceptions are when $H_j$ or $H_{j+1}$ are bottom blocks. Note that, because of the global independence constraint, these two cases are mutually exclusive. See \cref{updownup}.

    \begin{figure}
        \centering{
            \begin{tabular}{ccc}
                \includegraphics{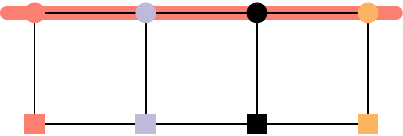} &
                \includegraphics{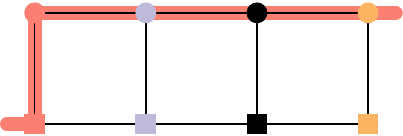} &
                \includegraphics{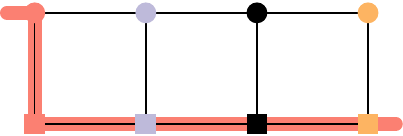}
            \end{tabular}
            \caption{The paths taken by $P$ through boring blocks: A top path, a downup path, and an updown path.}
            \label{updownup}
        }
    \end{figure}

    \begin{compactenum}
        \item When $H_j$ is a bottom block $P$ uses a path that begins at the bottom row of $Q_j$ but moves immediately to the top row of $Q_j$ and uses the entire path along the top row. We call this a \emph{downup} path.
        \item When $H_{j+1}$ is a bottom block, $P$ uses a path that begins at the top row of $Q_j$ and moves immediately to to the bottom row of $Q_j$ ad uses the entire path along the bottom row.  We call this a \emph{updown} path.
    \end{compactenum}
    This completely defines the path $P:=v_1,\ldots,v_{2m}$. All that remains is to argue that $\rho:=\phi_s(v_1),\ldots,\phi_s(v_{2m})$ is anagramish.

    Observe that the number of downup paths and the number of updown paths in $Q_0,\ldots,Q_{r-1}$ is exactly the same as the number of bottom blocks among $H_1,\ldots,H_{r-1}$ which is exactly $\beta$.  Similarly, the number of updown paths and downup paths in $Q_{r+1},\ldots,Q_{2r-1}$ is exactly $\beta$.  Now every path that is neither downup nor updown uses the top row.  This implies that the sequence of colours contributed to $\rho$ by  the intersection of $P$ with $Q_0,\ldots,Q_{r-1}$ is a permutation of the sequence of colours contributed to $\rho$ by the intersection of $P$ with $Q_{r+1},\ldots,Q_{2r-1}$.

    Finally, by construction, each pair of top and bottom blocks in $H_1,\ldots,H_{r-1}$ contributes exactly the same amount as a single matching zig-zag block in $H_{r+1},\ldots,H_{2r-1}$.  Specifically, if $x,y\in A_a$, and $z\in B_a$, $H_x$ is a top block, $H_y$ is a bottom block and $H_z$ is a zig-zag block, then the contributions of $P_x$ and $P_y$ to $\rho$ cancels out the contribution of $P_z$. After doing this cancellation exhaustively, all that remains are top blocks, which also cancel each other perfectly.  This completes the proof.
\end{proof}

For completeness, we wrap up the proof of \cref{main_vertex}:

\begin{proof}[Proof of \cref{main_vertex}]
    Assume for the sake of contradiction that there exists some $c\in\N$ such that $\afcn(G_n)\le c$ for each $n\in\N$.  With this assumption, \cref{near_anagram_graph} shows that for every $r_0\in\N$ there exists a string $s:=s_1,\ldots,s_{2r}\in\Sigma^*$ with $r\ge r_0$, $\tau(s)\le\epsilon r$, and such that $\phi_s$ is an anagram-free colouring $G_{n_s}$.

    \cref{anagramish_path} shows that, for any string $s_1,\ldots,s_{2r}\in\Sigma^*$ with $r\ge r_0$ and $\tau(s)\le\epsilon r$, the graph $G_{n_s}$ contain a path that is anagramish under $\phi_s$.  This is certainly a contradiction to the fact that $\phi_s$ is anagram-free colouring of $G_{n_s}$. We therefore conclude that, for every $c\in\N$ there exists an $n\in N$ such that $\afcn(G_n)> c$.
\end{proof}

% \todo[inline]{It would be nice to say something about the anagram-free chromatic index.  I'm pretty sure that it's constant for the $2\times n$ grid, but the result is not worth the effort.  More interesting would be a proof that its super-constant for the $2\times n$ grid with one diagonal in each square.  Almost everything we did above works except that we get hung up on a parity problem.  We can't do $2$ versus $1$ cancellation in $s$ but we can do $6$ versus $4$ cancellation. Everything works out except when the number of occurrences of some character is odd. In that case we get stuck with one uncancelled copy.}
%

\section{Proof of \cref{near_anagram_fourier}}
\label{near_anagram_proof}

% \todo[inline]{I suspect this lemma (and probably a better quantitative version) has a proof using Fourier approximation, but I don't know enough about Fourier approximation and my half-assed attempt to learn about it wasn't sufficient.}

All that remains is to prove \cref{near_anagram_fourier}, which we do now.

\begin{proof}[Proof of \cref{near_anagram_fourier}]
    Define an even-length string $t$ over the alphabet $\Sigma$ to be \emph{$a$-unbalanced} if $\tau_a(t)>\epsilon|t|/\ell$ and \emph{$a$-balanced} otherwise.  If $t$ is $a$-balanced for each $a\in\Sigma$ then $t$ is \emph{balanced}. Observe that, if $t$ is balanced then $\tau(t)\le \epsilon|t|$. A string is \emph{everywhere unbalanced} if it contains no balanced substring of length $r\ge r_0$. Our goal therefore is to show that there is an upper bound $n:=n(\ell,\epsilon,r_0)$ on the length of any $\ell$-periodic everywhere unbalanced string.

    Let $h$ be a positive integer (that determines $n$ and whose value will be discussed later), and let $n:=r_02^{h}$. Let $s$ be an $\ell$-periodic everywhere unbalanced string of length $n$ over the alphabet $\Sigma$.  The fact that $s$ is $\ell$-periodic, implies that the $|\Sigma|\le\ell$.  Assume, without loss of generality, that $r_0$ is a multiple of $\ell$.

    Consider the complete binary tree $T$ of height $h$ whose leaves, in order, are length-$r_0$ strings whose concatenation is $s$ and for which each internal node is the substring obtained by concatenating the node's left and right child.
    % See \cref{binary_tree}.
    Note that for each $v\in V(T)$ and each $a\in\Sigma$, the fact that $s$ is $\ell$-periodic and $r_0$ is multiple of $\ell$ implies that $\hist_a(v)\ge |v|/\ell$.

    % \begin{figure}
    %     \centering{
    %         \includegraphics[width=\textwidth]{figs/binary-tree-1}
    %     }
    %     \caption{Building a binary tree $T$ over the string $s:=4 3 5 3 2\ldots5 4 1 4$.}
    %     \label{binary_tree}
    % \end{figure}

  % For each node $v$ of $T$, let $h(v)$ denote the height of $v$'s subtree
  % and $s(v)=2^{h(v)}$ denote the length of the string $v$. For each
  % $i\in\Sigma$, let $w_i(v):=n_i(v)/s(v)$.  Note that $0\le w_i(v)\le
  % 1/2$ and that $\sum_{i\in\Sigma} w_i(v)=1$.  Furthermore, if $v$
  % has two children $x$ and $y$, then $w_i(v) = (w_i(x)+w_i(y))/2$.
  % The assumption that $s$ has no balanced substring implies $v$ is
  % $i$-unbalanced, for some $i\in\Sigma$.  Assign each internal node
  % $v$ of $T$ the \emph{label} $\ell(v):=\min\{i\in\Sigma: \mbox{$v$
  % is $i$-unbalanced}\}$.

  % \begin{figure}
  %   \centering{
  %      \includegraphics[width=\textwidth]{figs/binary-tree-3}
  %   }
  %   \todo[inline]{This is an old figure and needs updating.}
  %   \caption{The tree $T$ with the vector $(\hist_1(v),\ldots,\hist_5(v))$ for each
  %    internal node $v$.  The unbalanced index that defines $\ell(v)$ is highlited for each node.}
  %   \label{binary_tree_2}
  % \end{figure}

  For each $a\in\Sigma$, let $S_a:=\{v\in V(T): \text{$v$ is $a$-unbalanced}\}$. Since $s$ is everywhere unbalanced, $\bigcup_{a\in\Sigma} S_a=V(T)$. Therefore,
  \[
     (h+1)n = \sum_{v\in V(T)}|v|\le \sum_{a\in\Sigma} \sum_{v\in S_a} |v|
  \]
  and therefore, there exists some $a^*\in\Sigma$ such that $\sum_{v\in S_{a^*}}|v|\ge (h+1)n/|\Sigma| \ge (h+1)n/\ell$.   At this point we are primarily concerned with appearances of $a^*$, so let $X:=S_{a^*}$, and, for each node $v\in V(T)$, let $W(v):=\hist_{a^*}(v)$.

  For each non-leaf node $v$ of $T$, let $R(v)$ denote a child of $v$ such that $W(R(v))\le \tfrac12\cdot W(v)$.  (It is helpful to think of $T$ as being ordered so that each right child $y$ with sibling $x$ has $W(y)\le W(x)$.)  For a non-leaf node $v\in X$ the fact that $v$ is $a^*$-unbalanced implies that
  \[  W(R(v)) \le \tfrac{1}{2}\cdot W(v) - \tfrac{\epsilon}{2\ell}\cdot |v|
      \le (\tfrac12-\tfrac{\epsilon}{2\ell})W(v) \enspace .
  \]

  From this point on we use the following shorthands. For any $S\subseteq V(T)$, $L(S):=\sum_{v\in S}|v|$,
  % \todo{$L$ and $R$ are bad notation choices: $L$ for length, $R$ for right child}
  $W(S):=\sum_{v\in S}W(v)$, and $R(S)=\{R(v):v\in S\}$.  Summarizing, we have a complete binary tree $T$ of height $h$ and
  $X\subseteq V(T)$ with the following properties:
  \begin{compactenum}
    % \item For each $v\in V(T)$, $w(v) \le 1/2$.
    \item For each $v\in V(T)$,  $W(v)\ge |v|/\ell$.
    % \item For each internal node $v\in V(T)$ with children $x$ and $y$,
       % $w(v) = (w(x)+w(y))/2$.
    \item $L(X) \ge (h+1)n/\ell$.
    \item For each non-leaf node $v\in X$,
      $W(R(v)) \le (\tfrac{1}{2}-\tfrac{\epsilon}{2\ell})W(v)$.
  \end{compactenum}
  For each $i\in\{0,\ldots,h\}$, let $X_i\subseteq X$ denote the
  set of nodes $v\in X$ for which the path from the root of $T$ to $v$
  contains exactly $i$ nodes in $X$, excluding $v$.  See \cref{bigtree}. Observe that, since each node in $X_i$ has an ancestor in $X_{i-1}$,
  \[  n \ge L(X_0) \ge L(X_1) \ge \cdots\ge L(X_{h}) \enspace . \]

  We will show that there exists an integer $t:=t(\epsilon,\ell,r_0)$ such that, for each $i\in\{0,\ldots,h-t\}$,
  \begin{equation}
     L(X_{i+t}) \le (1-(1/2)^{t+1}) L(X_i) \enspace . \label{t}
  \end{equation}
  In this way,
  \begin{align*}
     (h+1)n/\ell
        \le L(X) & = \sum_{i=0}^{h} L(X_i) \\
           &\le \sum_{i=0}^{h} L(X_{t\floor{i/t}})
             & \text{(since $t\floor{i/t}\le i$)} \\
           & = t\cdot\sum_{i=0}^{h/t} L(X_{it})
             & \text{(for $h$ a multiple of $t$)} \\
           &\le t\cdot\sum_{i=0}^{\infty} (1-(1/2)^{t+1})^i L(X_0)
           & \text{(by \cref{t})}\\
           &\le tn\cdot \sum_{i=0}^{\infty} (1-(1/2)^{t+1})^i
           & \text{(since $|X_0|\le n$)} \\
           & = tn2^{t+1} \\
  \end{align*}
  which is a contradiction for sufficiently large $h$; in particular, for $h > \ell t2^{t+1}-1$.

  It remains to establish \cref{t}, which we do now.  Define $A_0:= X_i$ and, for each $j\ge 1$, define $A_j$ to be the subset of $X_{i+j}$ that are descendants of some node in $R(A_{j-1})$.  See \cref{bigtree}.  To upper bound $L(X_{i+t})$ observe that $X_{i+t}$ can be split into two sets $A_0'$ and $B$ defined as follows:  The nodes $A_0'$ do not have an ancestor in $R(A_0)$ and therefore $L(A_0')\le (1/2)L(A_0)$. The nodes in $B$ do have an ancestor in $R(A_0)$ and therefore have an ancestor in $A_1$.  Iterating this argument, we obtain
 \begin{align*}
      L(X_{i+t})
         &\le (1/2)L(A_0) + (1/2)L(A_1) + \cdots + (1/2)L(A_{t-1}) + L(A_t) \\
         &\le (1/2)L(A_0) + (1/4)L(A_0) + \cdots + (1/2)^t L(A_{0}) + L(A_t) \\
         &  = (1-(1/2)^t)L(A_0) + L(A_t) = (1-(1/2)^t)L(X_i) + L(A_t)  \enspace .
   \end{align*}
  So all that remains to establish \ref{t} is to prove that
  $L(A_t)\le (1/2)^{t+1}L(X_i)$.

  \begin{figure}
    \centering{
       \includegraphics[width=\textwidth]{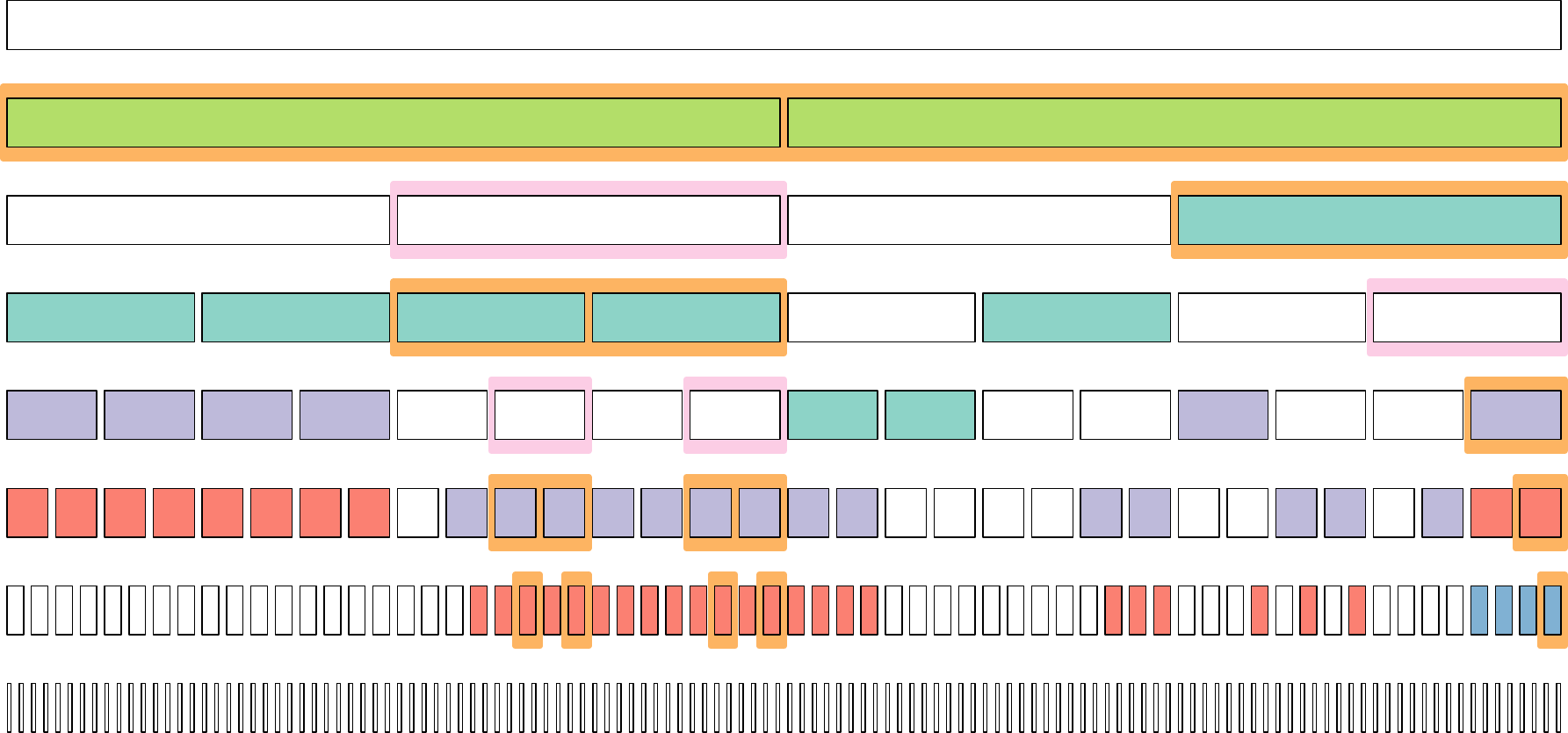}
    }
    \caption{The partitioning of $X$ into $X_0,\ldots,X_h$. Shaded
    nodes are in $X$ and all nodes in $X_i$ are shaded with the same
    colour.   Starting with $A_0=X_0$, the elements of $A_0,\ldots,A_h$
    are highlighted (in orange).  The elements of $R(A_0),\ldots,R(A_h)$ are also highlighted (in pink).}
    \label{bigtree}
  \end{figure}

  To to this, observe that, for each $j\in\{1,\ldots,t\}$,
  \begin{equation}
      W(A_j)
        \le W(R(A_{j-1}))
        \le (\tfrac12-\tfrac{\epsilon}{2\ell})\cdot W(A_{j-1})
  \end{equation}
  which implies
  \[
       W(A_t)
       \le (\tfrac12-\tfrac{\epsilon}{2\ell})^t W(A_0) \le (\tfrac12-\tfrac{\epsilon}{2\ell})^t\cdot L(A_0)
       =  (\tfrac12-\tfrac{\epsilon}{2\ell})^t\cdot L(X_i)
  \]
  Since $s$ is $\ell$-periodic,
  \[
        L(A_t)\le \ell\cdot W(A_t) \le \ell\cdot(\tfrac12-\tfrac{\epsilon}{2\ell})^t\cdot L(X_i) \le L(X_i)/2^{t+1}
  \]
  for $t = \ceil{\log(2\ell)/\log((1/(1-\tfrac{\epsilon}{\ell}))}$.
\end{proof}

\section{Reflections}
\label{reflections}

Although an explicit upper bound on $n:=n(\epsilon,\ell,r_0)$ could be extracted from the proof of \cref{near_anagram_fourier} it would likely be far from tight.  We suspect that there is a Fourier analytic proof that would give better quantitative bounds.  We have not pursued this, because we have no idea how to explicitly upper bound $\ell$, for reasons discussed in the next paragraph.

\cref{periodicity} and its proof give absolutely no clues to help find a concrete bound on $\ell$ or to find a minimal set $\Xi$. Indeed, for some choices of $P$, doing so can be a difficult problem.  Consider the example where $|\Sigma|=5$ and $P$ is predicate that tells whether or not its input is anagram-free. It is easy to see that this predicate $P$ satisfies (A1) and the result of \citet{pleasants:non-repetitive}, published in 1970, shows that this $P$ satisfies (A2).  The question of whether $|\Xi|=4$ or $|\Xi|=5$ is then the question of determining whether there exist arbitrarily long anagram-free strings on an alphabet of size $4$.  This was the open problem posed by \citet{erdos:some} in 1961 and again by \citet{brown:is} in 1971 and not resolved until 1992 when \citet{keranen:abelian,keranen:powerful} showed that the answer, in this case, is that $|\Xi|=4$.  However, if this were not the case, then determining $\ell$ would be the question of determining the length of the longest anagram-free string over an alphabet of size $4$.

Our proof uses \cref{periodicity} twice and each application uses a predicate $P$ that is considerably more complicated than asking if the input string is anagram-free. It seems unlikely that we will obtain concrete bounds upper bounds on $\ell$ as a function of $c$ except, possibly, through the use of computer search. The resulting value $\ell$ is used in the application of \cref{near_anagram_fourier} and also within the proof of \cref{anagramish_path}.

\bibliographystyle{plainurlnat}
\bibliography{af2t}

\begin{thebibliography}{13}
\providecommand{\natexlab}[1]{#1}
\providecommand{\url}[1]{\texttt{#1}}
\providecommand{\urlprefix}{URL }
\expandafter\ifx\csname urlstyle\endcsname\relax
  \providecommand{\doi}[1]{\href{https://dx.doi.org/#1}{\nolinkurl{doi:#1}}}\else
  \providecommand{\doi}[1]{\href{https://dx.doi.org/#1}{\nolinkurl{doi:#1}}}\fi
\providecommand{\eprint}[2][]{\url{#2}}

\bibitem[{Brown(1971)}]{brown:is}
T.~C. Brown.
\newblock Is there a sequence on four symbols in which no two adjacent segments
  are permutations of one another?
\newblock \emph{The American Mathematical Monthly}, 78(8):886--888, 1971.
\newblock \doi{10.1080/00029890.1971.11992892}.

\bibitem[{Carmi et~al.(2018)Carmi, Dujmović, and
  Morin}]{carmi.dujmovic.ea:anagram-free}
Paz Carmi, Vida Dujmović, and Pat Morin.
\newblock Anagram-free chromatic number is not pathwidth-bounded.
\newblock In Andreas Brandst{\"{a}}dt, Ekkehard K{\"{o}}hler, and Klaus Meer,
  editors, \emph{Graph-Theoretic Concepts in Computer Science - 44th
  International Workshop, {WG} 2018, Cottbus, Germany, June 27-29, 2018,
  Proceedings}, volume 11159 of \emph{Lecture Notes in Computer Science}, pages
  91--99. Springer, 2018.
\newblock \doi{10.1007/978-3-030-00256-5\_8}.

\bibitem[{Erdős(1961)}]{erdos:some}
P.~Erdős.
\newblock Some unsolved problems.
\newblock \emph{Magyar Tud Akad Mat Kutato Int Kozl}, 6:221--254, 1961.

\bibitem[{Evdokimov(1968{\natexlab{a}})}]{evdokimov:strongly}
A.~A. Evdokimov.
\newblock Strongly asymmetric sequences generated by finite number of symbols.
\newblock \emph{Doklady Akademii Nauk SSSR}, 179:1268–1271,
  1968{\natexlab{a}}.

\bibitem[{Evdokimov(1968{\natexlab{b}})}]{evdokimov:strongly2}
A.~A. Evdokimov.
\newblock Strongly asymmetric sequences generated by finite number of symbols.
\newblock \emph{Soviet Mathematics Doklady}, 9:536--539, 1968{\natexlab{b}}.

\bibitem[{Kamčev et~al.(2018)Kamčev, Łuczak, and
  Sudakov}]{kamcev.luczak.ea:anagram-free}
Nina Kamčev, Tomasz Łuczak, and Benny Sudakov.
\newblock Anagram-free colourings of graphs.
\newblock \emph{Comb. Probab. Comput.}, 27(4):623--642, 2018.
\newblock \doi{10.1017/S096354831700027X}.

\bibitem[{Ker{\"{a}}nen(1992)}]{keranen:abelian}
Veikko Ker{\"{a}}nen.
\newblock Abelian squares are avoidable on 4 letters.
\newblock In Werner Kuich, editor, \emph{Automata, Languages and Programming,
  19th International Colloquium, ICALP92, Vienna, Austria, July 13-17, 1992,
  Proceedings}, volume 623 of \emph{Lecture Notes in Computer Science}, pages
  41--52. Springer, 1992.
\newblock \doi{10.1007/3-540-55719-9\_62}.

\bibitem[{Ker{\"{a}}nen(2009)}]{keranen:powerful}
Veikko Ker{\"{a}}nen.
\newblock A powerful abelian square-free substitution over 4 letters.
\newblock \emph{Theor. Comput. Sci.}, 410(38-40):3893--3900, 2009.
\newblock \doi{10.1016/j.tcs.2009.05.027}.

\bibitem[{Nešetřil and de~Mendez(2006)}]{nesetril.ossona:tree-depth}
Jaroslav Nešetřil and Patrice~Ossona de~Mendez.
\newblock Tree-depth, subgraph coloring and homomorphism bounds.
\newblock \emph{Eur. J. Comb.}, 27(6):1022--1041, 2006.
\newblock \doi{10.1016/j.ejc.2005.01.010}.

\bibitem[{Pleasants(1970)}]{pleasants:non-repetitive}
P.~A.~B. Pleasants.
\newblock Non-repetitive sequences.
\newblock \emph{Proceedings of the Cambridge Philosophical Society},
  68:267--274, 1970.
\newblock \doi{10.1017/S0305004100046077}.

\bibitem[{Wilson(2019)}]{wilson:anagram-free}
Tim~E. Wilson.
\newblock \emph{Anagram-free Graph Colouring and Colour Schemes}.
\newblock Ph.D. thesis, Monash University, 2019.
\newblock \doi{10.26180/5c72eca26d5c7}.

\bibitem[{Wilson and Wood(2018{\natexlab{a}})}]{wilson.wood:anagram-free2}
Tim~E. Wilson and David~R. Wood.
\newblock Anagram-free colorings of graph subdivisions.
\newblock \emph{{SIAM} J. Discret. Math.}, 32(3):2346--2360,
  2018{\natexlab{a}}.
\newblock \doi{10.1137/17M1145574}.

\bibitem[{Wilson and Wood(2018{\natexlab{b}})}]{wilson.wood:anagram-free}
Tim~E. Wilson and David~R. Wood.
\newblock Anagram-free graph colouring.
\newblock \emph{Electron. J. Comb.}, 25(2):P2.20, 2018{\natexlab{b}}.
\newblock \doi{10.37236/6267}.

\end{thebibliography}

\end{document}